\documentclass[12pt]{amsart}
\usepackage{amsmath,amssymb,hyperref,geometry,graphicx,tikz,ytableau,fp,mathdots,amsthm}
\usetikzlibrary{cd}

\newtheorem{dummy}{}[section]
\newtheorem{theorem}[dummy]{Theorem}
\newtheorem*{ELSV}{ELSV Formula}
\newtheorem*{mainresult}{Main Result}
\newtheorem{corollary}[dummy]{Corollary}
\newtheorem{proposition}[dummy]{Proposition}

\newtheorem{lemma}[dummy]{Lemma}
\theoremstyle{definition}
\newtheorem{example}[dummy]{Example}
\newtheorem{remark}[dummy]{Remark}

\newtheoremstyle{theorem}
                 {10pt}
                 {0pt}
                 {}
                 {}
                 {\bfseries}
                 {}
                 {\newline}
                 {}
\theoremstyle{theorem}
\newtheorem*{remarks}{Remarks}

\newcommand{\Z}{\mathbb{Z}}
\renewcommand{\P}{\mathbb{P}}
\newcommand{\C}{\mathbb{C}}
\newcommand{\Q}{\mathbb{Q}}
\newcommand{\h}{\mathrm{h}}
\renewcommand{\t}{\mathrm{t}}
\renewcommand{\k}{\mathrm{k}}
\renewcommand{\v}{\mathrm{\mathbf{v}}}
\renewcommand{\P}{\mathcal{P}}

\makeatletter
\DeclareRobustCommand{\cev}[1]{
  {\mathpalette\do@cev{#1}}
}
\newcommand{\do@cev}[2]{
  \vbox{\offinterlineskip
    \sbox\z@{$\m@th#1 x$}
    \ialign{##\cr
      \hidewidth\reflectbox{$\m@th#1\vec{}\mkern4mu$}\hidewidth\cr
      \noalign{\kern-\ht\z@}
      $\m@th#1#2$\cr
    }
  }
}
\makeatother

\numberwithin{equation}{section}

\begin{document}

\title{Polynomiality of factorizations in reflection groups}
\author[E.~Polak]{Elzbieta Polak}
\address{Department of Mathematics, The University of Texas at Austin, 2515 Speedway Stop C1200, Austin, Texas 78712, USA}
\email{epolak@utexas.edu}
\author[D.~Ross]{Dustin Ross}
\address{Department of Mathematics, San Francisco State University, 1600 Holloway Avenue, San Francisco, California 94132, USA}
\email{rossd@sfsu.edu}

\begin{abstract}
We study the number of ways of factoring elements in the complex reflection groups $G(r,s,n)$ as products of reflections. We prove a result that compares factorization numbers in $G(r,s,n)$ to those in the symmetric group $S_n$, and we use this comparison, along with the ELSV formula, to deduce a polynomial structure for factorizations in $G(r,s,n)$.
\end{abstract}

\maketitle

\section{Introduction}

Given a set of generators $R$ of a multiplicative group $G$, we ask the natural enumerative question: How many ways can an element $\omega\in G$ be factored as a product of $m$ elements from $R$? Given $G$ and $R$, we denote these integer counts by $f_m^\omega\in\Z_{\geq 0}$. We seek to understand the structural properties of these numbers, especially in the setting of reflection groups. 

A special case of this setup is the classical problem of counting factorizations of a given permutation $\omega\in S_n$ as a product of transpositions. It is well known that such factorizations in $S_n$ are equivalent to degree-$n$ holomorphic maps from a complex curve to the projective line with ramification specified by $\omega$ over one point and simple ramification at $m$ additional points \cite{CavalieriMiles}. These counts are of long-standing interest in combinatorics, functional analysis, algebraic geometry, integrable systems, and physics, dating back at least to the pioneering work of Hurwitz in the late 19th century \cite{Hurwitz}. 

Hurwitz's foundational work suggested the importance of studying \emph{transitive} factorizations in $S_n$---that is, factorizations $\rho_m\cdots\rho_1=\omega$ where the subgroup generated by the transpositions $\rho_1,\dots,\rho_m$ acts transitively on $\{1,\dots,n\}$. In terms of maps of complex curves, transitivity is equivalent to the condition that the domain is topologically connected. Let $\widetilde f_m^\omega$ be the number of length-$m$ transitive factorizations of $\omega\in S_n$. One of the most remarkable results about these numbers is the celebrated ESLV formula, proved by Ekedahl, Lando, Shapiro, and Vainshtein, which relates the counts of transitive factorizations to intersection numbers on moduli spaces of curves.

\begin{ELSV}[\cite{ELSV}]
If $\omega\in S_n$ has cycle type $(n_1,\dots,n_\ell)$ and $g=\frac{1}{2}(m-n-\ell+2)$, then
\[
\widetilde f_m^\omega=m!\prod_{i=1}^\ell \frac{n_i^{n_i+1}}{n_i!}P_{g,\ell}(n_1,\dots,n_\ell)
\]
where $P_{g,\ell}$ is the symmetric polynomial
\[
P_{g,\ell}(x_1,\dots,x_\ell)=\int_{\mathcal{\overline{M}}_{g,\ell}}\frac{1-\lambda_1+\lambda_2-\cdots+(-1)^g\lambda_g}{(1-x_1\psi_1)\cdots(1-x_\ell\psi_\ell)}\in\Q[x_1,\dots,x_\ell].
\]
\end{ELSV}

The polynomial structure implied by the ELSV formula is truly striking and imposes a great deal of structure on the factorization numbers. 
Not only is each $P_{g,\ell}$ a symmetric polynomial, but the degrees of the nonzero terms lie in the interval $[2g-3+\ell,3g-3+\ell]$.  In other words, there is an effective bound on the number of nonzero coefficients in each $P_{g,\ell}$, and those finitely-many coefficients then determine the infinitely-many values of $\widetilde f_m^\omega$ obtained by fixing $g$ and $l$ while varying $m$ and $(n_1,\dots,n_\ell)$. Although this polynomial structure was proved by the ELSV formula, it had been discovered earlier by Goulden, Jackson, and Vainshtein \cite{GJV}.

It is often the case that results about symmetric groups are special cases of phenomena that hold more generally for reflection groups, and polynomiality is not an exception to this rule. In this paper, we generalize all aspects of the polynomial structure implied by the ELSV formula to the infinite family of complex reflection groups $G(r,s,n)$.



To state the main result, we give a brief overview of notation (see Section \ref{section:definitions} for precise definitions). We are interested in the complex reflection groups $G(r,s,n)$ where $r,s$, and $n$ are positive integers with $s\mid r$. Each of these groups is generated by a set of reflections $R\subseteq G(r,s,n)$, and for any $\omega\in G(r,s,n)$, we let $\widetilde f_m^\omega$ denote the number of \emph{connected} factorizations of $\omega$ into $m$ reflections. There is a natural group homomorphism $\pi:G(r,s,n)\rightarrow S_n$ and a function $\delta:G(r,s,n)\rightarrow\{0,1\}$; the main result is stated in terms of these functions.

\begin{mainresult}[Theorem \ref{theorem:mainresult}]
Fix $r,s\in\Z_{>0}$ such that $s\mid r$. For any $g,\ell\in\Z_{\geq 0}$, there exist two symmetric polynomials
\[
P_{g,\ell}^{0},P_{g,\ell}^{1}\in \Q[x_1,\dots,x_\ell]
\]
depending on $r$ and $s$ such that, if $\pi(\omega)$ has cycle type $(n_1,\dots,n_\ell)$ and $g=\frac{1}{2}(m-n-\ell+2)$, then
\[
\widetilde f_m^\omega=\frac{m!}{r^{n-1}}\prod_{i=1}^\ell \frac{n_i^{n_i+1}}{n_i!}P_{g,\ell}^{\delta(\omega)}(n_1,\dots,n_\ell).
\]
In addition, the nonzero terms in the polynomials $P_{g,\ell}^{i}$ all have degrees in the interval 
\[
[2g-3+\ell,3g-3+\ell].
\]
\end{mainresult}

\begin{remarks}\mbox{}
\vspace*{-\parsep}
\vspace*{-\baselineskip}
\begin{enumerate}
\item The ``connected'' condition for factorizations in $G(r,s,n)$ is a natural generalization of the ``transitive'' condition in $S_n$. It is shown in Proposition \ref{prop:connected} that connected factorization numbers in $G(r,s,n)$ determine all factorization numbers.

\item The proof of the main result expresses each $P_{g,\ell}^i$ as an explicit linear combination of the polynomials $P_{g',\ell}$ with $g'\leq g$, where the latter polynomials are those that appear in the ELSV formula. Thus, in cases where we have explicit formulas for $P_{g',\ell}$ for $g'\leq g$, we also obtain explicit formulas for $P_{g,\ell}^i$.

\item By the classification of Shephard and Todd \cite{ShephardTodd}, the infinite family $G(r,s,n)$ comprises all but 34 irreducible complex reflection groups. It is currently unclear whether one should expect a uniform polynomial structure that extends to the exceptional groups. The authors leave the investigation of such a generalization to future work.

\item Given the bounds on the degree of the polynomials $P_{g,\ell}^i$, one might hope that there is an interpretation of these polynomials in terms of intersection numbers on appropriate generalizations of the moduli spaces $\overline{\mathcal{M}}_{g,\ell}$. Finding such an interpretation would be very interesting, especially if it could be extended to the exceptional groups, and the authors also leave these investigations to future work.
\end{enumerate}
\end{remarks}

\subsection{Relation to previous work}

Many recent results in the literature have begun to uncover the structure inherent to factorizations in complex reflection groups \cite{BGJ,ChapuyStump,Michel,dHR,Douvropoulos,LewisMorales}. In much of the existing literature, the authors fix an element of a complex reflection group and compute an explicit formula for the generating series of all factorizations of that element. The formulas that have been found are quite compelling, especially the uniform formula discovered by Chapuy and Stump for factoring Coxeter elements in well-generated complex reflection groups \cite{ChapuyStump}, generalizing a result of Jackson that computes all factorizations of a long cycle in a symmetric group \cite{Jackson}.

The polynomial structure presented in this paper is somewhat orthogonal to the previous results. In particular, the previous results in the work cited above studied formulas for a fixed $\omega$ and varying $m$, which is equivalent to fixing $\ell$ while varying $g$. The polynomial structure, on the other hand, only arises when we fix $g$ and $\ell$ and we vary $\omega$, not just in a single group, but throughout all $G(r,s,n)$ with $n\geq \ell$.

Another distinction between this work and the previous papers is that most of the previous results use Burnside's character formula to study factorization numbers. Since our focus is on connected factorizations, these techniques are less immediately applicable.

While the polynomials in this paper are not given by explicit formulas, they provide a very general structural understanding of factorization numbers for complex reflection groups. Most of the previous results with explicit formulas study factorization numbers of Coxeter elements in well-generated complex reflection groups (Douvropoulos also extended these formulas to \emph{regular} elements \cite{Douvropoulos}). In the family $G(r,s,n)$, the only well-generated groups are those for which $s=1$ or $s=r$, and the Coxeter elements are very special, generalizing the long cycle in the case of $S_n$. Polynomiality, on the other hand, reveals a structure inherent to the collection of \emph{all} factorization numbers for \emph{all} groups $G(r,s,n)$.

However, if one requires explicit formulas for factoring specific elements in $G(r,s,n)$, then the methods of this paper are also useful. In particular, Corollary \ref{corollary:series} provides an explicit comparison between the factorization series of $\omega\in G(r,s,n)$ and that of $\pi(\omega)\in S_n$.  As an example application, we use the known formula for calculating factorizations of long cycles in $S_n$ to write a generating series for factoring long cycles in $G(r,s,n)$ (Corollary \ref{corollary:explicit}).

\subsection{Plan of the paper}

Section \ref{section:definitions} collects information about complex reflection groups and various types of factorizations. We begin in \ref{subsec:complex} with a review of complex reflection groups, describing in detail the groups $G(r,s,n)$ that are of primary interest in this work. In \ref{subsec:factorizations}, we define the particular factorization numbers $f_m^\omega$ that we study, along with a refinement $f_{m_1,m_2}^\omega$ that will be important. In \ref{subsec:graphs}, we introduce a useful graph-theoretic interpretation of these factorizations. We close Section \ref{section:definitions} with a discussion of connected factorization numbers $\widetilde f_m^\omega$ and $\widetilde f_{m_1,m_2}^\omega$, and we describe how to recover all factorizations from the connected ones.

Section \ref{sec:results} contains the main results of this paper. These results all follow from Theorem \ref{theorem:comparison}, which compares factorization numbers for $\omega\in G(r,s,n)$ with those of $\pi(\omega)\in S_n$. More specifically, Theorem \ref{theorem:comparison} expresses $\widetilde f_m^\omega$ as a linear combination of the numbers $\widetilde f_{m'}^{\pi(\omega)}$ with $m'\leq m$. From this comparison result and the polynomiality implied by the ELSV formula, it then follows that the factorization numbers $\widetilde f_m^\omega$ are determined by polynomials. That the bounds on the degree work out so nicely was a pleasant surprise, and is a result of the specific structure of the comparison in Theorem \ref{theorem:comparison}. We close Section \ref{sec:results} by providing an explicit comparison between factorization series of $\omega\in G(r,s,n)$ and $\pi(\omega)\in S_n$, and we use this to compute a formula for the factorization series of long cycles in $G(r,s,n)$.

\subsection{Acknowledgements}

The authors are grateful to Stanza Coffee on 16th Street in San Francisco for providing a pleasant environment in which to ponder polynomiality, and they are especially indebted to Luis for pouring the tastiest espressos that fueled their progress.

The second named author learned about the results of \cite{ChapuyStump} at the Workshop on Moduli Spaces, Integrable Systems, and Topological Recursions hosted by the Universit\'e de Montr\'eal in January 2016; he thanks the organizers for the invitation to participate and Guillaume Chapuy for his talk on the results of \cite{ChapuyStump}, which initiated this investigation.

This work was partially supported by the National Science Foundation (RUI DMS-2001439).

\section{Factorizations in complex reflection groups}\label{section:definitions}

In this section, we collect definitions and examples of complex reflection groups, with a special emphasis on the groups $G(r,s,n)$. We then describe the different types of factorizations that we are interested in, and we introduce a graph-theoretic interpretation of those factorizations.

\subsection{Complex reflections groups}\label{subsec:complex}

Let $V$ be a finite-dimensional complex vector space. A linear transformation $\rho\in\mathrm{GL(V)}$ is said to be a \emph{reflection of V} if it has finite order and if the fixed-point set
\[
\{\v\in V\;|\;\rho(\v)=\v\}
\]
is a complex hyperplane in $V$. A \emph{complex reflection group} is a finite subgroup $G\subset \mathrm{GL}(V)$ that is generated by reflections. 

The data of a complex reflection group is more than just an abstract group; the complex vector space $V$ and an embedding in $\mathrm{GL}(V)$ must also be specified. As opposed to reflections of real vector spaces, reflections of complex vector spaces do not necessarily have order two. A simple example of a complex reflection with order greater than two is multiplication by a primitive $r$th root of unity, with $r>0$, viewed as an element of $\mathrm{GL}(\C)$.

The next example provides a general description of all of the reflections that will be considered in this paper and establishes notation that will be used throughout.

\begin{example}\label{ex:reflections}
Let $V=\C^n$ and let $\zeta_r=\mathrm{e}^{2\pi\mathrm{i}/r}$ be a primitive $r$th root of unity. 
\begin{enumerate}
\item For each $1\leq i<j\leq n$ and $k/r\in\Q\cap[0,1)$, define the linear transformation $\sigma_{ij}^{k/r}\in\mathrm{GL}(V)$ to be the function that transposes the $i$th and $j$th coordinates then multiplies them by $\zeta_r^{-k}$ and $\zeta_r^{k}$, respectively:\footnote{Note that $\zeta_r^{\pm k}$ is independent of the representative we choose for the rational number $k/r$.}
\[
\sigma_{ij}^{k/r}(a_1,\dots,a_i,\dots,a_j,\dots,a_n)=(a_1,\dots,\zeta_r^{-k} a_j,\dots,\zeta_r^{k}a_i,\dots,a_n).
\]
The transformation $\sigma_{ij}^{k/r}$ is a reflection of $V$ because it has finite order $2$ and the fixed-point set is the hyperplane defined by the linear equation $x_j=\zeta_r^kx_i$. When $k/r=0$, we often omit it from the notation and write $\sigma_{ij}=\sigma_{ij}^0$.

\item For each $1\leq i\leq n$ and each $k/r\in\Q\cap(0,1)$, define the linear transformation $\tau_i^{k/r}\in\mathrm{GL}(V)$ to be the function that multiplies the $i$th coordinate by $\zeta_r^k$:
\[
\tau_i^{k/r}(a_1,\dots,a_i,\dots,a_n)=(a_1,\dots,\zeta_r^k a_i,\dots, a_n).
\]
The transformation $\tau_i^{k/r}$ is a reflection of $V$ because it has finite order equal to the smallest positive integer $d$ such that $r\mid dk$ and the fixed-point set is the hyperplane defined by $x_i=0$.
\end{enumerate}
\end{example}

The previous example provides us with a wealth of complex reflections. The next two examples describe the complex reflection groups that can be generated by sets of these reflections. We begin with the most classical example: the symmetric group.

\begin{example}\label{ex:Sn}
Let $V=\C^n$. The complex reflection group in $\mathrm{GL}(V)$ generated by 
\[
\left\{\sigma_{ij}\;|\; 1\leq i<j\leq n\right\}
\] 
is isomorphic to the symmetric group $S_n$, which is embedded in $\mathrm{GL}(V)$ as the set of linear transformations that permute the coordinates of $V$. Concretely, $S_n$ can be described as the set of \emph{permutation matrices}---the $n\times n$ matrices with a single $1$ in each row and column and zeros elsewhere. These matrices act on $V=\C^n$ by matrix multiplication on the left.
\end{example}

The class of complex reflection groups that are of primary interest in this work are a natural generalization of the symmetric group. They are described in the following example.

\begin{example}\label{ex:G(r,s,n)}
Let $V=\C^n$ and let $r$ and $s$ be positive integers such that $s\mid r$. The complex reflection group $G(r,s,n)\subset\mathrm{GL(V)}$ is the group generated by
\begin{equation}\label{eq:generators}
\left\{\sigma_{ij}^{k/r}\;\Big|\; {1\leq i<j\leq n \atop 0\leq k< r}\right\}\cup \left\{\tau_i^{sk/r}\;\Big|\; {1\leq i\leq n \atop 0< k< r/s}\right\}.
\end{equation}
Concretely, the elements of $G(r,s,n)$ are the $n\times n$ matrices (acting on $V$ by matrix multiplication on the left) that satisfy the following three conditions:
\begin{enumerate}
\item each row and column has a single nonzero entry,
\item each nonzero entry is an $r$th root of unity, and
\item the product of the nonzero entries is an $(r/s)$th root of unity.
\end{enumerate}
It can be checked that the generating set in \eqref{eq:generators} is a complete list of reflections in $G(r,s,n)$. We let $R$ denote this set and write
\[
R=R_1 \sqcup R_2,
\] 
where $R_1=\{\sigma_{ij}^{k/r}\}$ and $R_2=\{\tau_i^{sk/r}\}$. Notice that $R_2=\emptyset$ unless $s<r$.

As familiar special cases of $G(r,s,n)$, observe that $G(1,1,n)$ is the symmetric group $S_n$, while $G(r,s,1)=\mu_{r/s}$ is the group of $(r/s)$th roots of unity. It can also be shown that the dihedral groups are isomorphic to $G(r,r,2)$ for $r>2$. For example, the symmetries of the square are isomorphic to $G(4,4,2)$.
\end{example}

Irreducible complex reflection groups were classified by Shephard and Todd in \cite{ShephardTodd}. In addition to the infinite family $G(r,s,n)$ described in Example \ref{ex:G(r,s,n)}, Shephard and Todd described a list of 34 exceptional groups, and they showed that every complex reflection group can be decomposed uniquely as a product, each factor of which is either $G(r,s,n)$ for some $(r,s,n)$ or one of the exceptional groups. As the only infinite family in the classification, the groups of the form $G(r,s,n)$ play an especially important role in the theory of complex reflection groups.

There are two natural group homomorphisms between complex reflection groups that are important. The first is the homomorphism 
\[
\pi:G(r,s,n)\rightarrow S_n
\]
that replaces every nonzero entry in a matrix $\omega\in G(r,s,n)$ with $1$. The second is the homomorphism
\[
\phi:G(r,s,n)\rightarrow \mu_{r/s}
\]
that computes the product of all of the nonzero entries in a matrix $\omega\in G(r,s,n)$. The function $\delta$ appearing in the statement of polynomiality is related to the homomorphism $\phi$; it is defined by
\begin{align*}
\delta:G(r,s,n)&\rightarrow\{0,1\}\\
\omega&\mapsto\begin{cases}
1&\text{ if }\phi(\omega)=1,\\
0&\text{ if }\phi(\omega)\neq 1.
\end{cases}
\end{align*}

\subsection{Factorizations}\label{subsec:factorizations}

Given an element $\omega\in G(r,s,n)$, our goal is to study the number of ways it can be factored into reflections. More precisely, define the factorization set $F^\omega_m\subseteq R^m$ to be the set of ways to write $\omega$ as a product of $m$ reflections:
\[
F^\omega_m:=\left\{(\rho_1,\dots,\rho_m)\in R^m\;|\;\rho_m\cdots \rho_1=\omega\right\}.
\]
We are interested in the size of this factorization set:
\[
f^\omega_m:=|F^\omega_m|.
\]
Notice that the notation $(r,s,n)$ has been omitted from these definitions because it is encoded by the element $\omega$, which is always understood to be an element of some $G(r,s,n)$.

Given a factorization
\[
\rho_m\cdots \rho_1=\omega\in G(r,s,n),
\]
we obtain an identity in $S_n$ by applying the homomorphism $\pi$:
\[
\pi(\rho_m)\cdots\pi(\rho_1)=\pi(\omega)\in S_n.
\]
However, $\left(\pi(\rho_1),\dots,\pi(\rho_m)\right)$ is not necessarily an element of $F^{\pi(\omega)}_m$, because $\pi(\rho_j)$ will be the identity whenever $\rho_j\in R_2$. Thus, if we want to compare factorizations in $G(r,s,n)$ with factorizations in $S_n$, it makes sense to refine the factorizations by the number of factors that belong to $R_1$ and $R_2$.

For an element $\omega\in G(r,s,n)$ and nonnegative integers $m_1$ and $m_2$ such that $m=m_1+m_2$, define the refined factorization set as
\[
F^\omega_{m_1,m_2}:=\left\{(\rho_1,\dots,\rho_m)\in R^{m}\;\Big|\;{\rho_m\cdots \rho_1=\omega \atop |\{i\;|\;\rho_i\in R_j\}|=m_j}\right\}\subseteq F_m^\omega,
\]
and define
\[
f^\omega_{m_1,m_2}:=|F^\omega_{m_1,m_2}|.
\]
Notice that 
\[
f^\omega_{m}=\sum_{m_1+m_2=m}f_{m_1,m_2}^\omega.
\]
Applying the homomorphism $\pi$ to each factor then defines a function from the factorization set $F^\omega_{m_1,m_2}$ to the factorization set $F^{\pi(\omega)}_{m_1}$, which, by a slight abuse of notation, we also denote by $\pi$:
\begin{equation}\label{eq:map}
\pi:F^\omega_{m_1,m_2}\rightarrow F^{\pi(\omega)}_{m_1}.
\end{equation}
In order to compare the number of factorizations of $\omega\in G(r,s,n)$ to the number of factorizations of $\pi(\omega)\in S_n$, it suffices to understand the preimages of the function in \eqref{eq:map}.

\subsection{Factorization graphs}\label{subsec:graphs}

To assist in our study of the factorization sets introduced in the previous subsection, we introduce a combinatorial interpretation of those sets in terms of decorated graphs, and we prove several results concerning these graphs that will be useful in Section \ref{sec:results}.

For our purposes, a \emph{graph} on the vertex set $V=[n]=\{1,\dots,n\}$ consists of an ordered set of edges $E=(e_1,\dots,e_m)$ where each edge $e\in E$ is a multiset of order two: $e=\{i,j\}$ with $i,j\in [n]$. For clarity, it's worth noting that
\begin{itemize}
\item the set of edges is ordered, and
\item self-edges and multiple edges are allowed.
\end{itemize}
We often require subsets of edges to preserve the ordering, and we write $E'\preceq E$ to denote a subset of edges with the induced ordering. The ordered set of edges containing two distinct vertices is denoted $E_1\preceq E$ and the ordered set of self-edges is denoted $E_2\preceq E$.

The graphs that encode factorizations in $G(r,s,n)$ have an additional edge labeling. In particular, we decorate each edge $e\in E$ by an integer $\k(e)$ that satisfies the following constraints:
\[
\begin{cases} 
0\leq \k(e)< r &\text{ if }e\in E_1,\\
0< \k(e)< r/s &\text{ if }e\in E_2.
\end{cases}
 \]
 
We define $\Gamma_m(r,s,n)$ to be the set of graphs with $m$ ordered edges on the vertex set $[n]$ along with an edge labeling as above. If $r=s$, then $E_2=\emptyset$, because self-edges are impossible to label under the above constraints. Let $\Gamma_{m_1,m_2}(r,s,n)\subseteq\Gamma_m(r,s,n)$ denote the subset of graphs such that $|E_1|=m_1$ and $|E_2|=m_2$.

\begin{example}\label{runningexample}
The following is a depiction of a graph in $\Gamma_{5,2}(6,2,4)$. Each edge is the multiset containing the two vertices it attaches; for example, $e_5=\{3,4\}$ and $e_7=\{1,1\}$. To keep the graph uncluttered, the edge labels are listed to the right instead of along the edges.
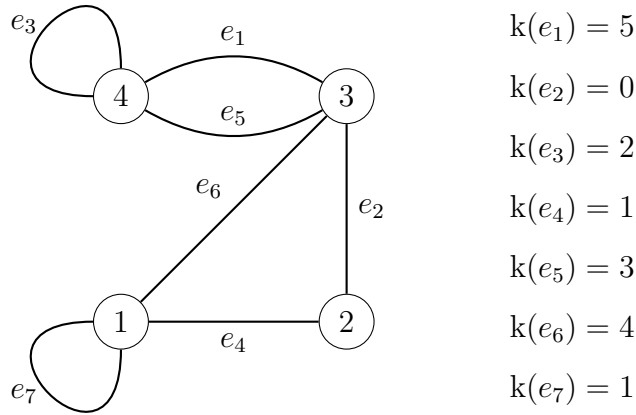
\vspace{-30bp}
\begin{center}
\tikz{
\node[circle, draw] (a) at (0,0) {$1$};
\node[circle, draw] (b) at (3,0) {$2$};
\node[circle, draw] (c) at (3,3) {$3$};
\node[circle, draw] (d) at (0,3) {$4$};
\draw[thick] (a) to node [below] {$e_4$} (b);
\draw[thick] (b) to node [right] {$e_2$} (c);
\draw[thick] (c) to [bend left] node [above] {$e_5$} (d);
\draw[thick] (a) to node [above left] {$e_6$} (c);
\draw[thick] (a) to [out=180,in=270,looseness=10] node [left] {$e_7$} (a); 
\draw[thick] (d) to [out=90,in=180,looseness=10] node [left] {$e_3$} (d); 
\draw[thick] (c) to [bend right] node [above] {$e_1$} (d);

\node at (6,3.9) {$\k(e_1)=5$};
\node at (6,3.1) {$\k(e_2)=0$};
\node at (6,2.3) {$\k(e_3)=2$};
\node at (6,1.5) {$\k(e_4)=1$};
\node at (6,.7) {$\k(e_5)=3$};
\node at (6,-.1) {$\k(e_6)=4$};
\node at (6,-.9) {$\k(e_7)=1$};
}
\end{center}
\vspace{-30bp}
\end{example}

The reason we have introduced decorated graphs is because each graph in $\Gamma_m(r,s,n)$ corresponds to an $m$-tuple of reflections $(\rho_1,\dots,\rho_m)\in R^m\subseteq G(r,s,n)^m$ in a natural way. Specifically, for each edge $e\in E$, we associate a reflection $\rho\in R$ via the rule
\begin{equation}\label{eq:bijection1}
\rho=\begin{cases}
\sigma_{ij}^{\k(e)/r} &\text{ if } e=\{i,j\}\in E_1\text{ with } i<j,\\
\tau_i^{s\k(e)/r} &\text{ if } e=\{i,i\}\in E_2.
\end{cases}
\end{equation}
Conversely, given an $m$-tuple of reflections $(\rho_1,\dots,\rho_m)\in R^m\subseteq G(r,s,n)^m$, we associate a decorated graph in $\Gamma_m(r,s,n)$ with edges and edge labels specified by
\begin{equation}\label{eq:bijection2}
(e,\k(e))=\begin{cases}
\left(\{i,j\},k\right) &\text{ if } \rho=\sigma_{ij}^{k/r}\in R_1,\\
\left(\{i,i\},k\right) &\text{ if } \rho=\tau_{i}^{sk/r}\in R_2.
\end{cases}
\end{equation}
The identifications in \eqref{eq:bijection1} and \eqref{eq:bijection2} are inverse to each other, and we henceforth use them to identify the set of graphs in $\Gamma_m(r,s,n)$ with the set of $m$-tuples of reflections in $G(r,s,n)$:
\begin{equation}\label{eq:bijection}
\Gamma_m(r,s,n)=R^m\subseteq G(r,s,n)^m.
\end{equation}
The subset $\Gamma_{m_1,m_2}(r,s,n)\subseteq\Gamma_m(r,s,n)=R^m$ corresponds to those $m$-tuples of reflections where $m_1$ of the reflections belong to $R_1$ and $m_2$ of the reflections belong to $R_2$.

\begin{example}\label{example:tupleofreflections}
The decorated graph in $\Gamma_{5,2}(6,2,4)$ depicted in Example \ref{runningexample} is associated to the following $7$-tuple of reflections in $G(6,2,4)$:
\[
\left(\sigma_{34}^5, \sigma_{23}^0,\tau_4^4, \sigma_{12}^1, \sigma_{34}^3, \sigma_{13}^4, \tau_1^2\right).
\]
\end{example}

Given a graph $\gamma\in\Gamma_m(r,s,n)$, let $\omega$ be the product of the corresponding $m$-tuple of reflections: $\omega=\rho_m\cdots\rho_1$. Then, by definition, $(\rho_1,\dots,\rho_m)\in F^\omega_m$. For each $\omega\in G(r,s,n)$, let $\Gamma^\omega_m\subseteq\Gamma_m(r,s,n)$ and  $\Gamma^\omega_{m_1,m_2}\subseteq\Gamma_{m_1,m_2}(r,s,n)$ denote the subsets of graphs whose corresponding reflections multiply to $\omega$. By definition, the identification \eqref{eq:bijection} induces an identification of these subsets with the appropriate factorization sets:
\vspace{-.8cm}
\begin{center}
\begin{align*}
\Gamma_m(r,s,n)  = & \;R^m \\
&\hspace{-1.5cm}\rotatebox[origin=c]{90}{$\subseteq$}  \hspace{1.4cm} \rotatebox[origin=c]{90}{$\subseteq$}\\
\Gamma^\omega_m \;\;\;=\;\; & \;F_m^\omega\\
&\hspace{-1.5cm}\rotatebox[origin=c]{90}{$\subseteq$}  \hspace{1.4cm} \rotatebox[origin=c]{90}{$\subseteq$}\\
\Gamma^\omega_{m_1,m_2} \;=\;\; & \;F_{m_1,m_2}^\omega.
\end{align*}
\end{center}

Our next goal is to better understand the subsets $\Gamma_m^\omega\subseteq \Gamma_m(r,s,n)$. In particular, given a graph in $\Gamma_m(r,s,n)$, we know that it belongs to $\Gamma_m^\omega$ for some $\omega$, and we would like to be able to describe $\omega$ in terms of the graph. For this, we turn to a discussion of \emph{ordered edge walks}.

Let $\gamma\in \Gamma_{m}(r,s,n)$ be a graph with edge set $E=(e_1,\dots,e_m)$. A \emph{directed edge} $\vec e=(i_0,i_1)$ consists of an edge $e=\{i_0,i_1\}\in E$ along with a choice of ordering of the two vertices. Notice that every edge in $E_1$ has two possible directions, while edges in $E_2$ have only one possible direction. Given a directed edge $\vec e=(i_0,i_1)$, we denote the \emph{head} and \emph{tail} by $\h(\vec e)=i_1$ and $\t(\vec e)=i_0$, respectively. A \emph{walk} in $\gamma$ is a set of directed edges $w=(\vec s_1,\dots,\vec s_\ell)$ such that $\h(\vec s_j)=\t(\vec s_{j+1})$ for $j=1,\dots,\ell-1$. A \emph{step} of the walk refers to a single directed edge $\vec s_j$. The \emph{start} and \emph{end} of the walk refer to $\t(\vec s_1)$ and $\h(\vec s_\ell)$, respectively. We use the following notation to denote walks:
\[
w=(i_0\stackrel{s_1}{\longrightarrow} i_1\stackrel{s_2}{\longrightarrow} \cdots \stackrel{s_\ell}{\longrightarrow} i_\ell),
\]
where $\h(\vec s_j)=i_j$ and $\t(\vec s_j)=i_{j-1}$ for all $j=1,\dots,\ell$. We say that a walk is \emph{ordered} if the order of the steps is consistent with the ordering of edges: $(s_1,\dots,s_\ell)\preceq E$.

Given a graph $\gamma\in \Gamma_{m}(r,s,n)$ and a vertex $i\in[n]$, there is a natural ordered walk $w_i=w_i(\gamma)$ that starts at vertex $i$ and sequentially walks along the edges $e_1,e_2,\dots,e_m$. More specifically, starting at vertex $i=i_0$, let $s_1=\{i_0,i_1\}$ be the first edge in $E$ that contains $i_0$, and define $\vec s_1=(i_0,i_1)$. Next, let $s_2=\{i_1,i_2\}$ be the first edge in $E$ that occurs after $s_1$ and contains $i_1$, and define $\vec s_1=(i_1,i_2)$. Continue in this way until, for some $i_\ell$, there does not exist an edge in $E$ that occurs after $s_\ell$ and contains $i_\ell$. When this happens, stop the recursion, and define
\[
w_i(\gamma)=(i_0\stackrel{s_1}{\longrightarrow} i_1\stackrel{s_2}{\longrightarrow} \cdots \stackrel{s_\ell}{\longrightarrow} i_\ell).
\]
If $i$ is not contained in any edges, then $w_i$ is the trivial walk that starts and ends at $i$ and does not have any steps.

\begin{example}\label{example:walks}
Consider the graph in Example \ref{runningexample}. By starting at each vertex and following the edges sequentially, the four ordered edge walks we obtain are:
\begin{align*}
w_1&=(1\stackrel{e_4}{\longrightarrow} 2);\\
w_2&=(2\stackrel{e_2}{\longrightarrow} 3\stackrel{e_5}{\longrightarrow} 4);\\
w_3&=(3\stackrel{e_1}{\longrightarrow} 4\stackrel{e_3}{\longrightarrow} 4\stackrel{e_5}{\longrightarrow} 3\stackrel{e_6}{\longrightarrow} 1\stackrel{e_7}{\longrightarrow} 1);\\
w_4&=(4\stackrel{e_1}{\longrightarrow} 3\stackrel{e_2}{\longrightarrow} 2\stackrel{e_4}{\longrightarrow} 1\stackrel{e_6}{\longrightarrow} 3).
\end{align*}
\end{example}

The next result is a useful observation about these ordered edge walks. It can be checked explicitly in Example \ref{example:walks}.

\begin{proposition}\label{prop:steps}
Let $\gamma\in\Gamma_m(r,s,n)$. If $\vec e$ is a directed edge in $\gamma$, then there is a unique $i$ such that $\vec e$ is a step on the walk $w_i(\gamma)$. Consequently,
\begin{enumerate}
\item every $e\in E_1$ occurs as a step on exactly two walks $w_i(\gamma)$ and $w_j(\gamma)$ with $i\neq j$, and 
\item every $e\in E_2$ occurs as a step on exactly one walk $w_i(\gamma)$.
\end{enumerate}
\end{proposition}

\begin{proof}
Given a graph $\gamma\in\Gamma_m(r,s,n)$ and a directed edge $\vec e$, let's walk backwards to find an $i$ such that $\vec e$ is a step on $w_i$. More specifically, let $s_{\text{-}1}=\{i_{\text{-}1},i_0\}$ be the last edge before $e$ that contains $i_0=\t(\vec e)$, and define $\vec s_{\text{-}1}=(i_{\text{-}1},i_0)$. Then let $s_{\text{-}2}=\{i_{\text{-}2},i_{\text{-}1}\}$ be the last edge before $s_{\text{-}1}$ that contains $i_{\text{-}1}$, and define $\vec s_2=(i_{\text{-}2},i_{\text{-}1})$. Continue in this way until, for some $i_{\text{-}\ell}$, there does not exist an edge before $s_{\text{-}\ell}$ that contains $i_{\text{-}\ell}$, and define $i=i_{\text{-}\ell}$. By construction, the walk $w_i$ is equal to
\[
w_i=(i_{\text{-}\ell}\stackrel{s_{\text{-}\ell}}{\longrightarrow} \cdots \stackrel{s_{\text{-}1}}{\longrightarrow} \t(\vec e)\stackrel{e}{\longrightarrow} \h(\vec e)\stackrel{s_1}{\longrightarrow} \cdots \stackrel{s_{\ell'}}{\longrightarrow} i_{\ell'})
\]
for some directed edges $\vec s_1,\dots,\vec s_{\ell'}$. Thus, $\vec e$ is a step on the walk $w_i$. 

To see that $\vec e$ cannot be a step on more than one walk $w_i$, it is enough to notice that, given a step on a walk $w_i$, both the preceding and the following step are uniquely determined. Thus, if two walks $w_i$ and $w_j$ share one step, then they must share all of their steps.
\end{proof}

So far, the ordered walks $w_i$ only take into account the ordering of the edges, but not the labels. We now take into consideration the edge labels. We define the \emph{weight} of a directed edge by
\[
\k(\vec e)=\begin{cases}
\k(e) & \text{ if } \h(\vec e)> \t(\vec e),\\
s\k(e) & \text{ if } \h(\vec e) = \t(\vec e),\\
-\k(e) & \text{ if } \h(\vec e) < \t(\vec e),
\end{cases}
\] 
where the three cases are referred to as \emph{up-steps}, \emph{loops}, and \emph{down-steps}, respectively. The \emph{weight} of a walk $w=(\vec s_1,\dots,\vec s_\ell)$ is defined by
\[
\k(w)=\sum_{j=1}^\ell \k(\vec s_j).
\]

\begin{example}\label{example:walklabels}
Consider the decorated graph in Example \ref{runningexample} with walks described in Example \ref{example:walks}. We compute the weights of these walks:
\begin{align*}
\k(w_1)&=1;\\
\k(w_2)&=0+3=3;\\
\k(w_3)&=5+2\cdot 2-3-4+2\cdot 1=4;\\
\k(w_4)&=-5-0-1+4=-2.
\end{align*}
\end{example}

Given the above definitions, we are now ready to characterize those graphs in $\Gamma_m(r,s,n)$ that correspond to factorizations of $\omega\in G(r,s,n)$. This is the content of the next result. 

\begin{proposition}\label{prop:graphs}
Let $\omega$ be an element of $G(r,s,n)$ and let $\gamma\in\Gamma_m(r,s,n)$ be a decorated graph. Then $\gamma\in \Gamma_m^\omega$ if and only if
\[
\omega(\v_i)=\zeta_r^{\k(w_i)}\v_{\h(w_i)}\;\;\text{ for all }\;\;i=1,\dots,n,
\]
where $\v_1,\dots,\v_n$ are the standard basis vectors of $\C^n$.
\end{proposition}

Before proving the proposition, we return one last time to our running example.

\begin{example}
Using the values of $\h(w_i)$ and $\k(w_i)$ computed in Examples \ref{example:walks} and \ref{example:walklabels}, we see that the decorated graph in $\Gamma_{5,2}(6,2,4)$ depicted in Example \ref{runningexample} corresponds to a length-$7$ factorization of
\[
\left(
\begin{array}{cccc}
0 & 0 & \zeta_6^4 & 0\\
\zeta_6 & 0 & 0 & 0\\
0 & 0 & 0 & \zeta_6^4\\
0 & \zeta_6^3 & 0 & 0\\
\end{array}
\right)\in G(6,2,4).
\]
This can be checked explicitly by multiplying the $7$ elements listed in Example \ref{example:tupleofreflections}.
\end{example}
\vspace{0bp}

\begin{proof}[Proof of Proposition \ref{prop:graphs}]
Let $\gamma\in\Gamma_m(r,s,n)$ be a decorated graph and consider the associated reflections $(\rho_1,\dots,\rho_m)$ defined in \eqref{eq:bijection1}. By definition, $\gamma\in \Gamma_m^\omega$ if and only if $\rho_m\cdots\rho_1=\omega$, and this equality can be verified by checking that both sides act the same way on the standard basis vectors. Thus, we must prove that
\[
\rho_m\cdots\rho_1(\v_i)=\zeta_r^{\k(w_i)}\v_{\h(w_i)}\;\;\text{ for all }\;\;i=1,\dots,n.
\]

Fix $i_0\in[n]$. In order to compute $\rho_m\cdots\rho_1(\v_{i_0})$, we begin by looking for the first reflection $r_1\in(\rho_1,\dots,\rho_l)$ of the form $r_1=\sigma_{i_0i_1}^{k/r}$, $r_1=\tau_{i_0}^{sk/r}$, or $r_1=\sigma_{i_1i_0}^{k/r}$. All other reflections fix $\v_{i_0}$, so they can be ignored for the purposes of this calculation. Notice that the edge $s_1$ in $\gamma$ corresponding to the reflection $r_1$ is the first step in the walk $w_{i_0}=(\vec s_1,\dots,\vec s_\ell)$, and the three possibilities for $r_1$ characterize whether $\vec s_1$ is an up-step, a loop, or a down-step. In each case, we compute
\begin{align*}
r_1(\v_{i_0})&=\begin{cases}
\zeta_r^k\v_{i_1} &\text{ if } r_1=\sigma_{i_0i_1}^{k/r}\\
\zeta_r^{sk}\v_{i_0} & \text{ if } r_1=\tau_{i_0}^{sk/r}\\
\zeta_r^{-k}\v_{i_1} & \text{ if }r_1=\sigma_{i_1i_0}^{k/r}
\end{cases}\\
&=\zeta_r^{\k(\vec s_1)}\v_{\h(\vec s_1)}.
\end{align*}
Next, look for the first reflection $r_2\in(\rho_1,\dots,\rho_\ell)$ occurring after $r_1$ of the form $r_2=\sigma_{\h(\vec s_1)i_2}^{k/r}$, $r_2=\tau_{\h(\vec s_1)}^{sk/r}$, or $r_2=\sigma_{i_2\h(\vec s_1)}^{k/r}$. By the same argument as above,
\begin{align*}
r_2(r_1(\v_{i_0}))&=r_2(\zeta_r^{\k(\vec s_1)}\v_{\h(\vec s_1)})\\
&=\zeta_r^{\k(\vec s_1)+\k(\vec s_2)}\v_{\h(\vec s_2)}.
\end{align*}
Continuing in this way, we construct a list of elements $(r_1,\dots,r_l)\preceq(\rho_1,\dots,\rho_m)$ such that
\begin{align*}
\rho_m\cdots\rho_1(\v_{i_0})&=r_\ell\cdots r_1(\v_{i_0})\\
&=\zeta_r^{\k(\vec s_1)+\cdots+\k(\vec s_\ell)}\v_{\h(\vec s_\ell)}\\
&=\zeta_r^{\k(w_{i_0})}\v_{\h(w_{i_0})}.
\end{align*}
\vspace{-30bp}

\end{proof}

\subsection{Connected factorizations}

Our main results in Section \ref{sec:results} are stated in terms of connected factorizations. In this subsection, we introduce connected factorizations and describe how they generalize the transitive factorizations in $S_n$. We also prove that every factorization number can be computed from the connected factorization numbers.

Let $\omega\in G(r,s,n)$. We say that a factorization 
\[
(\rho_1,\dots,\rho_m)\in F^\omega_m
\]
is \emph{connected} if the corresponding graph $\gamma\in\Gamma_m^\omega$ is connected, in the sense that there exists at least one walk between any two vertices. The following proposition shows how the notion of connected factorizations in $G(r,s,n)$ generalizes that of transitive factorization in $S_n$.\footnote{In \cite{LewisMorales}, Lewis and Morales studied a notion of transitive factorizations for $G(r,s,n)$ that also generalizes the notion of transitive factorizations in $S_n$. Their notion of transitive factorizations is more restrictive than the notion of connected factorizations studied here. On the other hand, the notion of connected factorizations studied here seems to agree with the notion of near admissible in work of Bini, Goulden, and Jackson on factorizations in hyperoctahedral groups \cite{BGJ}.}

\begin{proposition}
Let $\omega\in G(r,s,n)$. A factorization $(\rho_1,\dots,\rho_m)\in F^\omega_m$ is connected if and only if the subgroup generated by $\pi(\rho_1),\dots,\pi(\rho_m)\in S_n$ acts transitively on the standard basis vectors $\v_1,\dots,\v_n$.
\end{proposition}

\begin{proof}
Let $(\rho_1,\dots,\rho_m)\in F^\omega_m$ be a factorization with associated graph $\gamma\in\Gamma^\omega_m$.

Suppose that $\gamma$ is connected. To prove that the subgroup generated by $\{\pi(\rho_1),\dots,\pi(\rho_m)\}$ acts transitively on $\{\v_1,\dots,\v_n\}$, let $i,j\in[n]$ and, by connectedness, choose a walk between $i$ and $j$:
\begin{equation}\label{eq:walk}
w=(i=i_0\stackrel{s_1}{\longrightarrow} i_1\stackrel{s_2}{\longrightarrow} \cdots \stackrel{s_\ell}{\longrightarrow} i_\ell=j).
\end{equation}
Let $\gamma'\in\Gamma_\ell(r,s,n)$ be the graph with ordered edge set $(s_1,\dots,s_\ell)$ and notice that $w_i(\gamma')=w$, the walk described in \eqref{eq:walk}. Let $\{r_1,\dots,r_\ell\}\subseteq\{\rho_1,\dots,\rho_m\}$ be the subset of reflections corresponding to the edges $\{s_1,\dots,s_\ell\}$, and set $\omega'=r_\ell\cdots r_1$. By Proposition \ref{prop:graphs},
\[
\omega'(\v_i)=\zeta_r^k\v_j
\]
for some $k$, implying that
\[
\left(\pi(r_\ell)\cdots\pi(r_1)\right)(\v_i)=\pi(\omega'(\v_i))=\v_j.
\]
Thus, the subgroup generated by $\pi(\rho_1),\dots,\pi(\rho_m)$ acts transitively on $\{\v_1,\dots,\v_n\}$.

Conversely, assume that the subgroup generated by $\{\pi(\rho_1),\dots,\pi(\rho_m)\}\subseteq S_n$ acts transitively on $\{\v_1,\dots,\v_n\}$. Given $i,j\in[n]$, there is a subset $\{r_1,\dots,r_\ell\}\subseteq\{\rho_1,\dots,\rho_m\}$ such that
\begin{equation}\label{eq:walksnstuff}
\left(\pi(r_\ell)\cdots\pi(r_1)\right)(\v_i)=\v_j.
\end{equation}
Set $\omega'=r_\ell\cdots r_1$ and let $\gamma'\in\Gamma_\ell^{\omega'}$ be the graph associated to $(r_1,\dots,r_\ell)\in F^{\omega'}_\ell$. In order for \eqref{eq:walksnstuff} to be true, it must be the case that $\omega'(\v_i)=\zeta_r^k\v_j$ for some $k$. Therefore, by Proposition \ref{prop:graphs}, the ordered walk $w_i(\gamma')$ starts at vertex $i$ and ends at vertex $j$. By definition, the edges of $\gamma'$ are a subset of the edges in $\gamma$, so the ordered walk  $w_i(\gamma')$ corresponds to a (not-necessarily ordered) walk from $i$ to $j$ in the graph $\gamma$. Since $i$ and $j$ were arbitrary, this proves that $\gamma$ is a connected graph.
\end{proof}

We denote the sets of connected graphs by $\widetilde\Gamma_{m_1,m_2}\subseteq\widetilde \Gamma_m^\omega\subseteq \Gamma_m^\omega$, and  
we define the connected factorization numbers by
\[
\widetilde f_m^\omega:=|\widetilde \Gamma_m^\omega|\;\;\text{ and }\;\;\widetilde f_{m_1,m_2}^\omega:=|\widetilde \Gamma_{m_1,m_2}^\omega|.
\]

Since every graph decomposes uniquely as a disjoint union of connected graphs, every factorization number can be computed in terms of connected factorization numbers. To make this precise, we require a little more notation. Given a subset $I\subseteq [n]$, let $G(r,s,I)$ be the subset of $G(r,s,n)$ that fixes all standard basis vectors aside from those indexed by the elements of $I$. Given an element $\omega\in G(r,s,n)$, a \emph{partition} of $\omega$ consists of a set partition $[n]=I_1\sqcup\cdots\sqcup I_\ell$ along with elements $w_1\in G(r,s,I_1),\dots,w_\ell\in G(r,s,I_\ell)$ such that $\omega_1\cdots\omega_\ell=\omega.$ If $r=1$, then a partition of the permutation $\omega$ simply consists of all ways to group together its disjoint cycles. Let $\P(\omega)$ denote the set of partitions of $\omega$. The next result describes how to compute general factorization numbers from connected factorization numbers.

\begin{proposition}\label{prop:connected}
If $\omega\in G(r,s,n)$ and $m\in\Z_{\geq 0}$, then 
\[
f_m^\omega=\sum_{\substack{(\omega_1,\dots,\omega_\ell)\in\P(\omega) \\ m_1+\cdots+m_\ell=m}}{m \choose m_1,\dots,m_\ell}\widetilde f_{m_1}^{\omega_1}\cdots\widetilde f_{m_\ell}^{\omega_\ell}.
\]
\end{proposition}

\begin{proof}
Notice that every graph $\gamma\in \Gamma_m^\omega$ can be decomposed uniquely as a disjoint union of connected graphs $\gamma_1,\dots,\gamma_\ell$. If $\gamma_i$ has vertices $I_i\subseteq[n]$ and has $m_i$ edges, then $\gamma_i\in\widetilde\Gamma_{m_i}^{\omega_i}$ for some $\omega_i\in G(r,s,I_i)$. In addition, the vertex sets $\{I_i\}$ form a set partition of $[n]$ and the integers $\{m_i\}$ add up to $m$. Thus, there is a function
\[
\Gamma_m^\omega\rightarrow\bigsqcup_{\substack{(\omega_1,\dots,\omega_\ell)\in\P(\omega) \\ m_1+\cdots+m_\ell=m}}\widetilde\Gamma_{m_1}^{\omega_1}\times\cdots\times \widetilde\Gamma_{m_\ell}^{\omega_\ell}
\]
This function is surjective, but not injective. The number of graphs in the preimage of $(\gamma_1,\dots,\gamma_\ell)$ corresponds to the number of ways of choosing an ordering of all of the $m$ edges that is consistent with the ordering of the $m_i$ edges in each connected component $\gamma_i$. The number of ways of choosing such an ordering is counted by the multinomial
\[
{m \choose m_1,\dots,m_\ell},
\]
proving the formula in the proposition.
\end{proof}

\begin{remark}
While Proposition~\ref{prop:connected} shows that connected factorizations determine all factorizations in principle, it is quite difficult to implement this reconstruction in practice due to the complexity of computing the set $\P(\omega)$.
\end{remark}

\section{Comparison formula and polynomiality}\label{sec:results}

In this section, we prove the main comparison formula between connected factorizations in $G(r,s,n)$ and connected factorizations in $S_n$ (Theorem \ref{theorem:comparison}). We then use the comparison formula to prove polynomiality of factorizations in $G(r,s,n)$ (Theorem \ref{theorem:mainresult}). We close this section by reinterpreting the comparison formula in terms of exponential generating series (Corollary \ref{corollary:series}), then using this to compute the factorization series of all long cycles.

\subsection{Comparison formula}\label{subsec:comparison}

The next result, which is the technical heart of this paper,  utilizes the homomorphisms $\pi:G(r,s,n)\rightarrow S_n$ and $\phi:G(r,s,n)\rightarrow \mu_{r/s}$ and the function $\delta:G(r,s,n)\rightarrow\{0,1\}$ to describe an explicit comparison between the connected factorization numbers $\widetilde f_{m_1,m_2}^{\omega}$ associated to $\omega\in G(r,s,n)$ and the connected factorization numbers $\widetilde f_{m_1}^{\pi(\omega)}$ associated to $\pi(\omega)\in S_n$. 

\begin{theorem}\label{theorem:comparison}
For any $\omega\in G(r,s,n)$,
\begin{equation}\label{eq:comparison1}
\widetilde f_{m_1,m_2}^{\omega}=\left(r^{m_1-n+1}n^{m_2}{m_1+m_2 \choose m_1} f_{m_2}^{\phi(\omega)}\right)\widetilde f_{m_1}^{\pi(\omega)},
\end{equation}
where
\begin{equation}\label{eq:comparison2}
f_{m_2}^{\phi(\omega)}=\frac{1}{r/s}\left((r/s-1)^{m_2}-(-1)^{m_2}\right)+\delta(\omega)(-1)^{m_2}.
\end{equation}
\end{theorem}

Notice that the term $f_{m_2}^{\phi(\omega)}=|F_{m_2}^{\phi(\omega)}|$ is nothing more than the number of ways to write $\phi(\omega)\in\mu_{r/s}=G(r,s,1)$ as a product of nontrivial elements in the cyclic group. Since every factorization in the cyclic group is connected, we omit the tilde from the notation. 

\subsection{Proof of Theorem \ref{theorem:comparison}}
We begin by proving Equation \eqref{eq:comparison1}. When this is complete, we then turn to a proof of Equation \eqref{eq:comparison2}, which follows from the computation of cyclic factorization numbers in Proposition \ref{prop:cyclicfactors}.

To prove Equation \eqref{eq:comparison1}, let $\omega\in G(r,s,n)$ and consider the function
\[
\pi:\widetilde\Gamma_{m_1,m_2}^\omega\rightarrow \widetilde\Gamma_{m_1}^{\pi(\omega)},
\]
which forgets all self-edges and edge labels. Since 
\[
\widetilde f_{m_1,m_2}^{\omega}=|\widetilde\Gamma_{m_1,m_2}^\omega|\;\;\text{ and }\;\; \widetilde f_{m_1}^{\pi(\omega)}=|\widetilde\Gamma_{m_1}^{\pi(\omega)}|,
\]
it suffices to prove that, for any $\underline\gamma\in \widetilde\Gamma_{m_1}^{\pi(\omega)}$, the size of its preimage under $\pi$ is given by the following formula:
\begin{equation}\label{eq:preimage}
|\pi^{-1}(\underline\gamma)|=r^{m_1-n+1}n^{m_2}{m_1+m_2 \choose m_1} f_{m_2}^{\phi(\omega)}.
\end{equation}

Let $\underline\gamma\in \widetilde\Gamma_{m_1}^{\pi(\omega)}$; we begin by describing the graphs $\gamma\in\pi^{-1}(\underline\gamma)$. Such a graph $\gamma$ has $m_1+m_2$ ordered edges $E$. If we forget the self-edges, then we obtain $m_1$ edges that we denote $(e_1,\dots,e_{m_1})=E_1\preceq E$ corresponding to the $m_1$ ordered edges of $\underline\gamma$. In addition, each edge $e_i$ has a label that we denote $k_i$. If we forget the edges $e_1,\dots,e_{m_1}$, then we obtain $m_2$ ordered self-edges that we denote $(e_1',\dots,e_{m_2}')=E_2\preceq E$, and each self-edge $e_i'$ has a label that we denote $k_i'$. By Proposition \ref{prop:graphs}, the edge labels must satisfy the condition
\begin{equation}\label{eq:condition}
\omega(\mathrm{v}_i)=\zeta_r^{\k(w_i)}\mathrm{v}_{\h(w_i)}
\end{equation}
for all $i=1,\dots,n$, where $w_i=w_i(\gamma)$ is the ordered edge walk starting at vertex $i$. Our task is to count the number of ways to choose such edges and labels subject to the constraint \eqref{eq:condition}. Before working out the counting arguments carefully, we first summarize the three main points that will be proved.
\begin{itemize}
\item The number of ways to choose the edges in $E_2$ along with the ordering $E_2\preceq E$ is
\[
n^{m_2}{m_1+m_2 \choose m_1}.
\]
\item The number of ways to choose labels on the edges in $E_2$ is $f_{m_2}^{\phi(\omega)}$.
\item The number of ways to choose labels on the edges in $E_1$ is $r^{m_1-n+1}$.
\end{itemize}
We now justify each of these counts. In particular, this will prove Equation \eqref{eq:preimage}, and Equation \eqref{eq:comparison1} then follows.

Let's begin by counting the ways to choose $E_2$. Each self-edge in $E_2$ can by attached to any one of the $n$ nodes, which results in $n^{m_2}$ possibilities. In addition, we need to choose an inclusion $E_2\preceq E$, which can be thought of as choosing $m_2$ places in a line-up of $m_1+m_2$ possibilities; this choice is counted by the binomial coefficient ${m_1+m_2 \choose m_1}$. Together, the contribution of these choices to \eqref{eq:preimage} is a factor of 
\[
n^{m_2}{m_1+m_2 \choose m_1}.
\]

Now we count the number of ways to label all of the edges. Since $\phi(\omega)$ is the product of all of the nonzero entries in $\omega$, Equation \eqref{eq:condition} implies that the edge labels must satisfy
\begin{equation}\label{eq:updown}
\zeta_r^{\k(w_1)+\cdots+\k(w_n)}=\phi(\omega).
\end{equation}
Since each edge $e_i\in E_1$ occurs on exactly two of the $n$ ordered walks, once as an up-step and once as a down-step, it contributes $k_i-k_i=0$ to the exponent in \eqref{eq:updown}. On the other hand, each self-edge $e_i'\in E_2$ occurs as a loop on exactly one walk and contributes $sk_i'$ to the exponent. Thus, the condition \eqref{eq:updown} is equivalent to
\begin{equation}\label{eq:looplabels}
\zeta_{r/s}^{k_1'}\cdots\zeta_{r/s}^{k_{m_2}'}=\phi(\omega).
\end{equation}
In other words, the labels $k_i'$ on the self-edges must be chosen subject to the condition \eqref{eq:looplabels}, and the number of ways to do this is precisely counted by $f_{m_2}^{\phi(\omega)}$.

We have now chosen everything except for the labels $k_1,\dots,k_{m_1}$ on the edges in $E_1$, and it remains to prove that, given the above choices, there are exactly $r^{m_1-n+1}$ ways to choose these labels. To accomplish this, we use the following lemma.

\begin{lemma}\label{lemma:choices}
If $\gamma\in \widetilde\Gamma_{m}(r,s,n)$ is a connected graph, then there exists a set of edges $\{f_1,\dots,f_{n-1}\}\subseteq E_1$ and a labeling of the vertices $[n]=\{v_0,\dots,v_{n-1}\}$ such that, for every $\ell=1,\dots,n-1$,
\begin{enumerate}
\item $f_\ell$ is an edge on the walk $w_{v_\ell}$, and
\item there exists $j<\ell$ such that $f_\ell$ is also an edge on the walk $w_{v_j}$.
\end{enumerate}
\end{lemma}

Before proving Lemma \ref{lemma:choices}, let's use it to count the number of ways to choose the remaining labels $k_1,\dots,k_{m_1}$. Choose a specific subset of edges $\{f_1,\dots,f_{n-1}\}\subseteq E$ and a labeling of vertices $[n]=\{v_0,\dots,v_{n-1}\}$ satisfying the conditions of Lemma \ref{lemma:choices}, and then consider any choice of edge labeling of the edges in $\{e_1,\dots,e_{m_1}\}\setminus\{f_1,\dots,f_{n-1}\}$; notice that there are $r^{m_1-n+1}$ such choices. Given any such choice, we now prove that there exists a unique way to label the remaining edges $f_1,\dots,f_{n-1}$ such that \eqref{eq:condition} holds for all $i=1,\dots,n$. 

First, notice that $f_{n-1}$ is the only unlabeled edge on $w_{v_{n-1}}$. Therefore, the constraint \eqref{eq:condition} for $i=v_{n-1}$ uniquely determines the label on $f_{n-1}$. After fixing the label on $f_{n-1}$, the only possible unlabeled edge on $w_{v_{n-2}}$ is $f_{n-2}$, so \eqref{eq:condition} with $i=v_{n-2}$ uniquely determines this label. Continuing this way in decreasing order, we see that the label on $f_\ell$ is uniquely determined by \eqref{eq:condition} with $i=v_\ell$ for all $\ell=n-1,\dots,1$. 

The choices we made in the previous paragraph for the labels on $f_1,\dots,f_{n-1}$ ensure that the constraint \eqref{eq:condition} holds for $i\in\{v_1,\dots,v_{n-1}\}$, but it would be natural to worry about whether the constraint also holds for $i=v_0$. Not to fret---we know that
\[
\omega(\v_{v_0})=\zeta_r^k\v_{\h(w_{v_0})}
\]
for some $k$, and using the validity of \eqref{eq:condition} for $i\in\{v_1,\dots,v_{n-1}\}$ along with condition \eqref{eq:updown}, we compute that
\[
\zeta_r^{k+\k(w_{v_1})+\cdots+\k(w_{v_{n-1}})}=\phi(\omega)=\zeta_r^{\k(w_{v_0})+\k(w_{v_1})+\cdots+\k(w_{v_{n-1}})},
\]
from which it follows that $\zeta_r^k=\zeta_r^{\k(w_{v_0})}$, proving that \eqref{eq:condition} holds for $i=v_0$. Thus, every one of the $r^{m_1-n+1}$ choices for the labels in $E_1\setminus\{f_1,\dots,f_{n-1}\}$ can be extended uniquely to a labeling of the edges $E_1$ that satisfies \eqref{eq:condition}, proving that the total number of ways to label the edges in $E_1$ is $r^{m_1-n+1}$.

This concludes our counting arguments, and the only task that remains is to prove the lemma.

\begin{proof}[Proof of Lemma \ref{lemma:choices}]
Let $\gamma\in \widetilde\Gamma_{m}(r,s,n)$ be a connected graph and choose a single vertex $v_0\in[n]$. We recursively define vertices $v_1,\dots,v_{n-1}$ and edges $f_1,\dots,f_{n-1}$ that satisfy the two conditions in the lemma. Suppose we are given $v_0,\dots,v_\ell$ and $f_1,\dots,f_\ell$ satisfying the conditions of the lemma (if $\ell=0$, then the conditions of the lemma are vacuous). If $\ell=n-1$, then we are done. If not, we claim (proved below) that there exists an edge $f_{\ell+1}\in E_1\setminus\{f_1,\dots,f_\ell\}$ such that $f_{\ell+1}$ is a step in exactly one of the walks $w_{v_0},\dots,w_{v_\ell}$. Then Proposition \ref{prop:steps} implies that there must be some other vertex $v_{\ell+1}\in[n]\setminus\{v_0,\dots,v_\ell\}$ such that $f_{\ell+1}$ is also a step in $v_{\ell+1}$. We then conclude the recursive step by noticing that $v_0,\dots,v_{\ell+1}$ and $f_1,\dots,f_{\ell+1}$ satisfy the two conditions in the lemma.

To prove the claim in the previous paragraph, suppose towards a contradiction that every edge in $E_1$ that occurs as a step on one of the walks $w_{v_0},\dots,w_{v_\ell}$ actually occurs as a step on two of them. Since $\ell<n-1$, we know that $[n]\setminus\{v_0,\dots,v_\ell\}\neq\emptyset$, and for any $i\in[n]\setminus\{v_0,\dots,v_\ell\}$, we know from Proposition \ref{prop:steps} that the walk $w_i$ cannot share any edges in common with $w_{v_0},\dots,w_{v_\ell}$. Since $\gamma$ is connected, each walk has at least one edge, and it follows that there must be at least one edge in $\gamma$ that is not a step in any of the walks $w_{v_0},\dots,w_{v_\ell}$. Using again that $\gamma$ is connected, it follows that there must be such an edge in $E_1$ that shares a vertex with at least one of the walks $w_{v_0},\dots,w_{v_\ell}$. Choose a vertex $i\in[n]$ such that $i$ is a vertex in at least one of the walk $w_{v_0},\dots,w_{v_\ell}$ and such that there exists an edge $e=\{i,j\}\in E_1$ containing $i$ that is not a step in any of these walks.

Let $E(i)\preceq E_1$ be the ordered set of edges that contain $i$ and let $E(i)'\preceq E(i)$ be those edges that are steps in one (and thus, by assumption, two) of the walks $w_{v_0},\dots,w_{v_\ell}$, with complement $E(i)''=E(i)\setminus E(i)'$. By the choice of $i$ in the previous paragraph, both $E(i)'$ and $E(i)''$ are nonempty. If all of the vertices in $E(i)'$ come before the vertices in $E(i)''$, then the unique walk in $w_{v_0},\dots,w_{v_\ell}$ that approaches $i$ along the last edge in $E(i)'$ must eventually depart from $i$ along the first edge in $E(i)''$, a contradiction of the definition of $E(i)''$ (notice that, before departing $i$ for another vertex, the walk might loop along any number of edges in $E_2$). Similarly, if all of the edges in $E(i)'$ come after the vertices in $E(i)''$, then the unique walk that does not belong to $w_{v_0},\dots,w_{v_\ell}$ and that approaches $i$ along the last edge of $E(i)''$ must eventually depart along the first edge of $E(i)'$, which already belongs to two distinct walks, contradicting Proposition \ref{prop:steps}. Finally, in the remaining cases, we can always find $f_1,f_2\in E(i)'$ such that the set of edges in $E(i)''$ between $f_1$ and $f_2$ is nonempty. In this case, the unique walk in $w_{v_0},\dots,w_{v_\ell}$ that approaches $i$ along $f_1$ must depart along the first edge in $E(i)''$ that appears after $f_1$, a contradiction of the definition of $E(i)''$. 

Since every case leads to a contradiction, we conclude that we can always choose $f_{\ell+1}$ as in the first paragraph of the proof, and the proof of the lemma is complete.
\end{proof}

We have now proved Equation \eqref{eq:comparison1}; to finish the proof of Theorem \ref{theorem:comparison}, it therefore remains to prove Equation \eqref{eq:comparison2}. In the next proposition, we accomplish this by computing the cyclic factorization numbers $f_m^\kappa$ for any $\kappa\in\mu_{r}$. This result was essentially proved by Chapuy and Stump in \cite{ChapuyStump}, though they only considered the case where $\kappa$ is a generator. The proof presented here is a modification of theirs.

\begin{proposition}\label{prop:cyclicfactors}
For any integer $r\geq 2$ and element $\kappa\in\mu_r$,
\[
f_m^\kappa=\begin{cases}
\frac{1}{r}\left((r-1)^m-(-1)^m\right) & \text{ if } \kappa\neq 1,\\
\frac{1}{r}\left((r-1)^{m}-(-1)^{m}\right)+(-1)^{m} & \text{ if } \kappa=1.
\end{cases}
\]
\end{proposition}

\begin{proof}
Choose $m-1$ nontrivial elements $\zeta_r^{k_1},\dots,\zeta_r^{k_{m-1}}\in\mu_r$ and notice that their product $\zeta_r^{k_{m-1}}\cdots\zeta_r^{k_1}$ can be extended to a factorization of $\kappa$ into $m$ nontrivial elements
\[
\zeta_r^{k_m}\zeta_r^{k_{m-1}}\cdots\zeta_r^{k_1}=\kappa
\]
if and only if $\zeta_r^{k_{m-1}}\cdots\zeta_r^{k_{1}}\neq \kappa$. In other words, the number of factorizations into $m$ nontrivial elements is the total number of products of $m-1$ nontrivial elements minus those that multiply to $\kappa$. Thus, we obtain the following recursion:
\begin{equation}\label{eq:recursion}
f_m^\kappa=(r-1)^{m-1}-f_{m-1}^\kappa.
\end{equation}
It is straightforward to check that both of the formulas in the statement of the proposition satisfy the recursion in \eqref{eq:recursion}. The necessity for the two different formulas arises from the different initial values: 
\[
f_0^\kappa=\begin{cases}
0 &\text{ if } \kappa\neq 1,\\
1 &\text{ if } \kappa=1.
\end{cases}
\]
The validity of both of these initial values can also be checked directly from the formulas appearing in the statement of the proposition.
\end{proof}

\subsection{Polynomiality}\label{subsec:polynomiality}

The polynomial structure of all factorization numbers of $G(r,s,n)$ is now a direct application of Theorem \ref{theorem:comparison}.

\begin{theorem}\label{theorem:mainresult}
Fix $r,s\in\Z_{>0}$ such that $s\mid r$. For any $g,\ell\in\Z_{\geq 0}$, there exist two symmetric polynomials
\[
P_{g,\ell}^0,P_{g,\ell}^1\in \Q[x_1,\dots,x_\ell]
\]
depending on $r$ and $s$ such that, if $\pi(\omega)$ has cycle type $(n_1,\dots,n_\ell)$ and $g=\frac{1}{2}(m-n-\ell+2)$, then
\[
\widetilde f_m^\omega=\frac{m!}{r^{n-1}}\prod_{i=1}^\ell \frac{n_i^{n_i+1}}{n_i!}P_{g,\ell}^{\delta(\omega)}(n_1,\dots,n_\ell).
\]
In addition, the nonzero terms in the polynomials $P_{g,\ell}^i$ all have degrees in the interval 
\[
[2g-3+\ell,3g-3+\ell].
\]
\end{theorem}

\begin{proof}
Applying Theorem \ref{theorem:comparison}, we have
\[
\widetilde f_m^\omega=\sum_{m_1=0}^m\left(r^{m_1-n+1}n^{m-m_1}{m \choose m_1} f_{m-m_1}^{\phi(\omega)}\right)\widetilde f_{m_1}^{\pi(\omega)}.
\]
If $\pi(\omega)$ has cycle type $(n_1,\dots,n_\ell)$, the ELSV formula then implies that
\[
\widetilde f_m^\omega=\sum_{m_1=0}^m\left(r^{m_1-n+1}n^{m-m_1}{m \choose m_1} f_{m-m_1}^{\phi(\omega)}\right)m_1!\prod_{i=1}^\ell \frac{n_i^{n_i+1}}{n_i!}P_{g(m_1),\ell}(n_1,\dots,n_\ell)
\]
where
\[
g(m_1)=\frac{1}{2}(m_1-n-\ell+2)
\]
and the nonzero terms in $P_{g(m_1),\ell}$ have degrees in the interval 
\[
[2g(m_1)-3+\ell,3g(m_1)-3+\ell].
\]
Reorganizing terms, we see that
\[
\widetilde f_m^\omega=\frac{m!}{r^{n-1}}\prod_{i=1}^\ell  \frac{n_i^{n_i+1}}{n_i!}\sum_{m_1=0}^m\left(r^{m_1} \frac{f_{m-m_1}^{\phi(\omega)}}{(m-m_1)!}\right)n^{m-m_1}P_{g(m_1),\ell}(n_1,\dots,n_\ell),
\]
and we define
\[
P_{g,\ell}^{\delta(\omega)}(x_1,\dots,x_\ell)=\sum_{m_1=0}^m\left(r^{m_1} \frac{f_{m-m_1}^{\phi(\omega)}}{(m-m_1)!}\right)(x_1+\cdots+x_\ell)^{m-m_1}P_{g(m_1),\ell}(x_1,\dots,x_\ell).
\]
Observe that the degree of the nonzero terms in the $m_i$th summand are at least
\[
m-m_i+2g(m_1)-3+\ell=2g-3+\ell
\]
and at most
\begin{align*}
m-m_i+3g(m_1)-3+\ell&=m-m_1+\frac{3}{2}(m_1-n-l+2)-3+\ell\\
&\leq \frac{3}{2}(m-n-l+2)-3+\ell\\
&=3g-3+\ell,
\end{align*}
where the inequality in the second line uses the fact that $m_1\leq m$.
\end{proof}

\subsection{Generating series}\label{subsec:series}

For any element $\omega\in G(r,s,n)$, define the generating series of factorizations and of connected factorizations, respectively, by
\[
f^\omega(x)=\sum_{m\geq 0}f_{m}^\omega\frac{x^m}{m!} \;\; \text{ and }\;\;\widetilde f^\omega(x)=\sum_{m\geq 0}\widetilde f_{m}^\omega\frac{x^m}{m!}  
\]
where $x$ is a formal variable. The next result shows how Theorem \ref{theorem:comparison} can be interpreted in terms of these generating series.

\begin{corollary}\label{corollary:series}
For any $\omega\in G(r,s,n)$,
\[
\widetilde f^\omega(x)=\frac{1}{r^{n-1}}f^{\phi(\omega)}(nx)\widetilde f^{\pi(\omega)}(rx)
\]
where
\[
f^{\phi(\omega)}(x)=\frac{1}{r/s}\left(e^{(r/s-1)x}-e^{-x} \right)+\delta(\omega)e^{-x}.
\]
\end{corollary}

\begin{proof}
Applying Theorem \ref{theorem:comparison}, we compute
\begin{align*}
\widetilde f^\omega(x)&=\sum_{m\geq 0}\widetilde f_{m}^\omega\frac{x^m}{m!}\\
&=\sum_{m_1,m_2\geq 0}\widetilde f_{m_1,m_2}^\omega\frac{x^{m_1+m_2}}{(m_1+m_2)!}\\
&=\sum_{m_1,m_2\geq 0}\left(r^{m_1-n+1}n^{m_2}{m_1+m_2 \choose m_1} f_{m_2}^{\phi(\omega)}\right)\widetilde f_{m_1}^{\pi(\omega)}\frac{x^{m_1+m_2}}{(m_1+m_2)!}\\
&=\frac{1}{r^{n-1}}\sum_{m_1,m_2\geq 0}\left(f_{m_2}^{\phi(\omega)}\frac{(nx)^{m_2}}{m_2!} \right)\left(\widetilde f_{m_1}^{\pi(\omega)}\frac{(rx)^{m_1}}{m_1!} \right)\\
&=\frac{1}{r^{n-1}}f^{\phi(\omega)}(nx)\widetilde f^{\pi(\omega)}(rx),
\end{align*}
proving the first equation in the theorem. The second equation follows from Proposition \ref{prop:cyclicfactors} along with the Taylor series expression for the exponential function
\[
e^x=\sum_{m\geq 0}\frac{x^m}{m!}.
\vspace{-20bp}
\]
\end{proof}

If $\omega\in S_n$ is a long cycle, then Jackson computed an explicit formula for the factorization series of $\omega$ in \cite{Jackson}:
\[
f^\omega(x)=\frac{1}{n!}\left(e^{x\frac{n}{2}}-e^{-x\frac{n}{2}} \right)^{n-1}.
\]
We've omitted the tilde because all factorizations of the long cycle are connected. Using this, we obtain the following generalization to long cycles in $G(r,s,n)$.

\begin{corollary}\label{corollary:explicit}
If $\omega\in G(r,s,n)$ such that $\pi(\omega)$ is a long cycle, then
\[
f^\omega(x)=\frac{1}{n!r^{n-1}}f^{\phi(\omega)}(nx)\left(e^{x\frac{rn}{2}}-e^{-x\frac{rn}{2}} \right)^{n-1}
\]
where
\[
f^{\phi(\omega)}(x)=\frac{1}{r/s}\left(e^{(r/s-1)x}-e^{-x} \right)+\delta(\omega)e^{-x}.
\]
\end{corollary}

As a special case, Corollary \ref{corollary:explicit} recovers the main result of Chapuy and Stump \cite{ChapuyStump} for the group $G(r,1,n)$. More generally, Corollary \ref{corollary:series} should be thought of as a reduction from $G(r,s,n)$ to $S_n$---it computes an explicit formula for $\tilde f^\omega(x)$ whenever we happen to know an explicit formula for $\tilde f^{\pi(\omega)}(x)$.

\bibliographystyle{alpha}
\bibliography{references}

\end{document}